\documentclass{elsarticle} 
\usepackage{amsmath,amsthm}
\usepackage{amssymb,amscd,latexsym}
\usepackage{boxedminipage}
\usepackage{eepic}
\usepackage{epic}
\usepackage{wrapfig}

\newcommand{\erase}[1]{}
\theoremstyle{remark}
 
\newtheorem{theorem}{Theorem}[section]

\newtheorem{proposition}[theorem]{Proposition}

\newtheorem{corollary}[theorem]{Corollary}

\newtheorem{lemma}[theorem]{Lemma}
\numberwithin{equation}{section}
\newcommand{\bp}{\begin{pmatrix}}
\newcommand{\ep}{\end{pmatrix}}
\newcommand{\bps}{\begin{smallmatrix}}
\newcommand{\eps}{\end{smallmatrix}}

\def\Z{{\mathbb Z}}

\def \0{{\bf 0}}
\def \1{{\bf 1}}

\def \mf#1#2#3#4{
\xymatrix{{#1}\  \ar@<0.4ex>[r]^{{#2}} & \ {#4}
\ar@<0.4ex>[l]^{{#3}}
}
}

\def \mfs#1#2#3#4{\!
\xymatrix@C=1.5em{{#1} \! \ar@<0.2ex>[r]^{{#2}} & \! {#4}
\ar@<0.2ex>[l]^{{#3}}
}
\!}

\def \mfl#1#2#3#4{
\xymatrix@C=2.6em{{#1}\  \ar@<0.4ex>[r]^{{#2}} &\  {#4}
\ar@<0.2ex>[l]^{{#3}}
}
}

\def \mfss#1#2#3#4{\!
\xymatrix@C=1.5em{{#1} \ar@<0.3ex>[r]^{{#2}} & {#4}
\ar@<0.3ex>[l]^{{#3}}
}
\!}

\newcommand{\deeq}{\mathbin{\hbox{$=$ \lower 1.7pt\rlap{\hskip -8.5pt .}}}} 

\begin{document}
\begin{frontmatter}
\title{On the derivative at $t = 1$ of the skew-growth functions for Artin monoids.} 
\author{Tadashi Ishibe}

\begin{abstract}
Let $G_{M}^{+}$ be the Artin monoid of finite type generated by the letters $a_i, i\in I$ with respect to a Coxeter matrix $M$ that is equipped with the degree map $\deg\!:\!G_{M}^{+} \to\!\Z_{\ge0}$ defined by assigning to each equivalence class of words the length of the words, and let $N_{M, \deg}(t)\!:=\!\!\sum_{J \subset I}(-1)^{\#J} t^{\deg(\Delta_{J})}$ be the skew-growth function, where the summation index $J$ runs over all subsets of $I$ and $\Delta_{J}$ is the fundamental element in $G_{M}^{+}$ associated to the set $J$. In this article, we will calculate the derivative at $t = 1$ of the polynomial $N_{M, \deg}(t)$. As a result, we show that the polynomial $N_{M, \deg}(t)$ has a simple root at $t = 1$.
\end{abstract}
\end{frontmatter}
  








\section{Introduction} 
Let $G_{M}^{+}$ be the Artin monoid of finite type (\cite{[B-S]}\S1) generated by the letters $a_i, i\in I$ with respect to a Coxeter matrix $M$ (\cite{[B]}). Due to the homogeneity of the defining relations in $G_{M}^{+}$, we naturally define a map $\deg\!:\!G_{M}^{+} \to\!\Z_{\ge0}$ defined by assigning to each equivalence class of words the length of the words. The \emph{spherical growth function} for the monoid $G_{M}^{+}$ is defined as 
\[
P_{G_{M}^{+}, \deg}(t)\!:=\!\!\sum_{u \in G_{M}^{+}}\!t^{\deg(u)}.
\]
In \cite{[A-N]}\cite{[Bro]}\cite{[S1]}, they show that the inversion function $P_{G_{M}^{+}, \deg}(t)^{-1}$ is given by the following function, called the \emph{skew-growth function}, 
\[
N_{M, \deg}(t):= \sum_{J \subset I}(-1)^{\#J} t^{\deg(\Delta_{J})},
\]
where the summation index $J$ runs over all subsets of $I$ and $\Delta_{J}$ is the fundamental element in $G_{M}^{+}$ associated to the set $J$ (\cite{[B-S]}\S5). That has been investigated by several authors (\cite{[A-N]}\cite{[B]}\cite{[Bro]}\cite{[D]}\cite{[I1]}\cite{[I2]}\cite{[S1]}\cite{[S2]}\cite{[S3]}\cite{[S4]}\cite{[X]}). In \cite{[B-S]}\S4, the authors show that the monoid $G_{M}^{+}$ satisfies the LCM condition (i.e. any two elements $\alpha$ and $\beta$ in it admit the left (resp.~right) least common multiple). By using this property, for a subset $J \subset I$, they defined  the fundamental element $\Delta_{J}$, as the right least common multiple of all the letters $a_i, i\in J$. In \cite{[S1]}\S4, it is observed that the polynomial $N_{M, \deg}(t)$ has a simple root at $t = 1$. In this article, we will calculate the derivative at $t = 1$ of the polynomial $N_{M, \deg}(t)$. As a result, we show that the polynomial $N_{M, \deg}(t)$ has a simple root at $t = 1$.\par
Our main theorem is the following.
\begin{theorem}
{\it For a Coxeter matrix $M$, the derivative at $t = 1$ of the polynomial $N_{M, \deg}(t)$ is given by the following list:\\
\noindent
\,\,\,\,$A_{l\ge1}$\,: \,\,\, $N'_{M, \deg}(1) = (-1)^{l}$,\,\,\,\,\,\,\,\,\,\,\,\,\,\,\,\,\,\,\,\, $E_{8}$\,: \,\,\,\,\,\,\,\,\,\,\, $N'_{M, \deg}(1) = 44$,\\
\,\,\,\,$B_{l\ge2}$\,: \,\,\, $N'_{M, \deg}(1) = (-1)^{l}l$,\,\,\,\,\,\,\,\,\,\,\,\,\,\,\,\,\,\,\,\,\,$F_{4}$\,: \,\,\,\,\,\,\,\,\,\,\,\,\,\,$N'_{M, \deg}(1) = 10$,\\
\,\,\,\,$D_{l\ge4}$\,: \,\,\, $N'_{M, \deg}(1) = (-1)^{l}(l-2)$,\,\,\,\,\,$H_{3}$\,: \,\,\,\,\,\,\,\,\,\,\, $N'_{M, \deg}(1) = -8$,\\
\,\,\,\,$E_{6}$\,: \,\,\,\,\,\,\,\,\, $N'_{M, \deg}(1) = 7$,\,\,\,\,\,\,\,\,\,\,\,\,\,\,\,\,\,\,\,\,\,\,\,\,\,\,\,\,\,\,\,\,$H_{4}$\,: \,\,\,\,\,\,\,\,\,\,\,\,\,\,$N'_{M, \deg}(1) = 42$,\\
\,\,\,\,$E_{7}$\,: \,\,\,\,\,\,\,\,\, $N'_{M, \deg}(1) = -16$,\,\,\,\,\,\,\,\,\,\,\,\,\,\,\,\,\,\,\,\,\,\,\,\,$I_{2}(p\ge5)$: $N'_{M, \deg}(1) = p-2$.\\

}
\end{theorem}
   
The above statement can be verified by hand calculation for the types $E_{6}, E_{7}, E_{8}, F_{4}, H_{3}, H_{4}$ and $I_{2}(p\ge5)$. In \S3, we will prove Theorem 1.1 for the type $A_{l}$. By using the results in \S3, we will prove Theorem 1.1 for the type $B_{l}$ and $D_{l}$ in \S4, \S5.
As a corollary of Theorem 1.1, we obtain the following.
\begin{corollary}
{\it The polynomial $N_{M, \deg}(t)$ has a simple root at $t = 1$.
 }
\end{corollary}
\section{Preliminary results}
Let $l$ be a positive integer and let $I = \{1, 2, \ldots , l \}$. The Coxeter matrix $M = (m(\alpha, \beta))_{\alpha, \beta \in I}$ of the type $X_{l} \in \{ A_{l}, B_{l}, D_{l} \}$ is given by the following list. For the type $A_{l}$, we give
\[
m(\alpha, \beta) = \left \{ \begin{array}{lll}
 1 & \mbox{if $\alpha = \beta$}  \\
3  & \mbox{if  $|\alpha - \beta| = 1$\,\,\,\,\,\,\,\,\,\,\,\,\,\,\,\,\,\,\,\,\,\,\,\,\,\,\,\,\,\,\,\,\,\,\,\,\,\,\,\,\,\,\,\,\,\,\,\,\,\,\,\,} \\
   2 & \mbox{if $|\alpha - \beta| > 1$} 
          \end{array}
\right.
\]
For the type $B_{l}$
\footnote
{ For the type $B_{l}$, we adopt for convenience the different definition of the Coxeter matrix from that in \cite{[B]}.}, we give
\[
m(\alpha, \beta) = \left \{ \begin{array}{lll}
 1 & \mbox{if $\alpha = \beta$}  \\
3  & \mbox{if  $|\alpha - \beta| = 1$ and $\alpha + \beta > 3$\,\,\,\,\,\,\,\,\,\,\,\,} \\
   2 & \mbox{if $|\alpha - \beta| > 1$} \\
 4 & \mbox{if  $\alpha + \beta = 3$} 
          \end{array}
\right.
\]
For the type $D_{l}$, we give
\[
m(\alpha, \beta) = \left \{ \begin{array}{lll}
 1 & \mbox{if $\alpha = \beta$}  \\
2 & \mbox{if $|\alpha - \beta| > 1$ and $\alpha + \beta \not= 2l-2$} \\
 2 & \mbox{if $\alpha + \beta = 2l-1$} \\
3  & \mbox{if  $|\alpha - \beta| = 1$ and $\alpha + \beta < 2l-1$} \\
 3 & \mbox{if $\alpha + \beta = 2l-2$ and $\alpha \not= \beta$} \\
          \end{array}
\right.
\]
We simply write the polynomial $N_{M, \deg}(t)$ by $N_{X_{l}}(t)$. Namely, we put
\[
 N_{X_{l}}(t) := \sum_{J \subset I}(-1)^{\#J} t^{\deg(\Delta_{X_{l}, J})},
\]
where $\Delta_{X_{l}, J}$ is the fundamental element in the Artin monoid $G_{M}^{+}$ associated to the set $J$. Moreover, for a non-negative integer $j \in  \{\, 0, \ldots, l \,\}$ we put
\[
N_{X_{l}, j}(t) := \sum_{J \subset I,\, \#J = j} t^{\deg(\Delta_{X_{l}, J})},\,\, C_{X_{l}, j}:=\left.\frac{\mathrm{d}N_{X_{l}, j}(t)}{\mathrm{d}t}\right|_{t=1}.
\]
Therefore, we have the following equations:
\[
N'_{X_{l}}(1) = \sum_{j=1}^{l}(-1)^{j}C_{X_{l}, j} ,\,\, C_{X_{l}, j}= \sum_{J \subset I,\, \#J = j} \deg(\Delta_{X_{l}, J}).
\]
To a Coxeter matrix $M = (m(\alpha, \beta))_{\alpha, \beta \in I}$ of the type $X_{l}$, we attach a Coxeter graph $\Gamma_{X_{l}}$ whose vertices are indexed by the set $I$ and two vertices $\alpha$ and $\beta$ are connected by an edge iff $m(\alpha, \beta)\ge 3$. For a subset $J \subset I$, we associate a full subgraph $\Gamma_{X_{l}}(J)$, whose vertices are indexed by the set $J$. The edge is labeled by $m(\alpha, \beta)$ (omitted if $m(\alpha, \beta)=3$). We note that $\Gamma_{X_{l}}(I)$ corresponds to the graph $\Gamma_{X_{l}}$. For a subgraph $\Gamma_{X_{l}}(J)$ of $\Gamma_{X_{l}}$, we write the number of connected components of $\Gamma_{X_{l}}(J)$ by $k_{X_{l}}(J)$. Let $\Gamma_{X_{l}}(J)$ be a full subgraph of $\Gamma_{X_{l}}$ with $k$-connected components $\Gamma_{X_{l}}(J_1), \Gamma_{X_{l}}(J_2), \ldots, \Gamma_{X_{l}}(J_k)$.
Then, we write
\[
\Gamma_{X_{l}}(J) = \bigsqcup_{i=1}^{k} \Gamma_{X_{l}}(J_i) .
\]
We recall a fact from \cite{[B-S]}.
\begin{proposition}
{\it For a subset $J \subset I$, we suppose that the full subgraph  $\Gamma_{X_{l}}(J)$ has a decomposition $\Gamma_{X_{l}}(J) = \bigsqcup_{i=1}^{k} \Gamma_{X_{l}}(J_i)$. Then:  \\
 (1)  For $1 \leq i < j \leq k$, $\Delta_{X_{l}, J_i}$ and $\Delta_{X_{l}, J_j}$commute with each other. \\
 (2)  Then, the fundamental element $\Delta_{X_{l}, J}$ can be written as a product of the fundamental elements $\Delta_{X_{l}, J_1}, \ldots, \Delta_{X_{l}, J_k}$:
\[
\Delta_{X_{l}, J} \deeq \Delta_{X_{l}, J_1} \cdots \Delta_{X_{l}, J_k}.
\]
}
\end{proposition}
\noindent
Since the map $\deg$ is an additive map, we can compute 
\begin{equation}
\deg(\Delta_{X_{l}, J}) = \sum_{i=1}^{k} \deg(\Delta_{X_{l}, J_i}). 
\end{equation}

\section{Proof of the type $A_{l}$}
Let $l$ be a positive integer and let $I = \{1, 2, \ldots , l \}$. In this section, we will prove Theorem 1.1 for the type $A_{l}$. First, we have a remark on $k_{A_{l}}(J)$.
\begin{proposition}
{\it For a subset $J \subset I$, we put $j :=\#J$. Then:\\
(1) If the number $j$ satifies an inequality $1 \leq j \leq \lceil \frac{l}{2} \rceil$, then the number of connected components $k_{A_{l}}(J)$ can run over from $1$ to $j$.\\
(2)If the number $j$ satifies an inequality $j > \lceil \frac{l}{2} \rceil$, then the number of connected components $k_{A_{l}}(J)$ can run over from $1$ to $l-j+1$.} 
\end{proposition}
\noindent
We put $\beta_{l, j} :=\mathrm{min}\{ j, l-j+1 \}$. Then, the summary of Proposition 3.1 is that the number of connected components $k_{A_{l}}(J)$ can run over from $1$ to $\beta_{l, j}$. For two positive integers $j$ and $k$ with $j \leq l$ and $k \leq \beta_{l, j}$, we put
\[
 N^{(k)}_{A_{l}, j}(t) :=\sum_{J \subset I,\, \#J = j,\, k_{A_{l}}(J)=k} t^{\deg(\Delta_{A_{l}, J})},\,\, C^{(k)}_{A_l, j}:= \left.\frac{\mathrm{d}N^{(k)}_{A_{l}, j}(t)}{\mathrm{d}t}\right|_{t=1}.
\]
Therefore, we have the following equation:
\[
C^{(k)}_{A_{l}, j}= \sum_{J \subset I,\, \#J = j,\, k_{A_{l}}(J)=k} \deg(\Delta_{A_{l}, J}).
\]

\noindent
By definition, we have
\[
C_{A_{l}, j} = \sum_{k=1}^{\beta_{l, j}} C^{(k)}_{A_{l}, j}.
\]
We recall a fact from \cite{[B-S]}.
\begin{proposition}
{\it For a subset $J \subset I$, we suppose that the full subgraph  $\Gamma_{A_{l}}(J)$ is connected. Then, the degree $\deg(\Delta_{A_{l}, J})$ of the fundamental element is given by
\[
\deg(\Delta_{A_{l}, J}) = \binom {\#(J)+1}{2}.
\]
}
\end{proposition}
\noindent
From the equation (2.1), we easily show the following formula.
\begin{proposition}
{\it For a subset $J \subset I$, we suppose that the full subgraph  $\Gamma_{A_{l}}(J)$ has a decomposition $\Gamma_{A_{l}}(J) = \bigsqcup_{i=1}^{k} \Gamma_{A_{l}}(J_i)$. Then, the degree of the fundamental element $\Delta_{A_{l}, J}$ can be written as 
\[
\deg(\Delta_{A_{l}, J}) = \sum_{i=1}^{k} \deg(\Delta_{A_{l}, J_i}) = \sum_{i=1}^{k} \binom {\#(J_i)+1}{2}. 
\]
}
\end{proposition}
\noindent

\begin{proposition}
{\it Let  $j$ and $k$ be two positive integers with $j \leq l$ and $k \leq \beta_{l, j}$. For given positive integers $\tau_1, \ldots, \tau_k$ with $\sum_{i=1}^{k}\tau_{i} = j$, we define the set $S_{j, (\tau_1, \ldots, \tau_k)}$ by
\begin{equation*}
\left\{  J \subset I \left|
\begin{array}{l}
\#J = j, \Gamma_{A_{l}}(J) = \bigsqcup_{i=1}^{k} \Gamma_{A_{l}}(J_i)\,\,\mathrm{with}\,\,\mathrm{min}(J_{1}) < \cdots < \mathrm{min}(J_{k})\\
\mathrm{s.t.}\,\, \#J_i=\tau_i (i=1, \ldots, k) \\
\end{array}
\right.\right\}
\end{equation*}
Then, we have the following equation
\[
\#S_{j, (\tau_1, \ldots, \tau_k)} = \binom {l-j+1}{k}.
\]
}
\end{proposition}
\noindent
We remark that the result does not depend on the choice of positive integers $\tau_1, \ldots, \tau_k$. Hence, the number $\binom {l-j+1}{k}$ devides the number $C^{(k)}_{A_l, j}$. Then, we define the number $\widetilde{C}^{(k)}_{A_l, j}$ by the equation
\[
C^{(k)}_{A_l, j} = \widetilde{C}^{(k)}_{A_l, j}\cdot \binom {l-j+1}{k}.
\]
For two positive integers $j, k$ with $k \leq j$, we put 
\[
T_{k, j}:= \{ (\tau_{1}, \ldots ,\tau_{k}) \in \Z^{k}_{>0} | \sum_{i=1}^{k}\tau_{i} = j \}.
\] 
From the Proposition 3.3, we have
\[
 \widetilde{C}^{(k)}_{A_l, j} = \sum_{\,\,\,\,\,\,(\tau_{1}, \ldots ,\tau_{k}) \in T_{k, j}} \sum_{i=1}^{k}\binom {\tau_{i}+1}{2}. 
\]
\begin{lemma}
{\it Let  $j$ and $k$ be two positive integers with $j \leq l$ and $k \leq \beta_{l, j}$. Then, the following equation $\mathrm{E}_{j, k}$ holds.
\[
 \widetilde{C}^{(k)}_{A_l, j} = k \binom {j+1}{k+1}. 
\]
}
\end{lemma}
\begin{proof}
We will show the general equation $\mathrm{E}_{j, k}$ by double induction. First, for $k=1$, the subgraph $\Gamma_{A_{l}}(J)$ is connected. Hence, we easily compute $\deg(\Delta_{A_{l}, J}) = \binom {j+1}{2}$. Therefore, we say the equation $\mathrm{E}_{j, 1}$ is true. Next, for induction hypothesis, we assume\\
$(\mathrm{A})$\,\,$\mathrm{E}_{j, k}$ is true for $j = 1, \ldots, r$ and arbitrary $k$, \\
and\\
$(\mathrm{B})$\,\,$\mathrm{E}_{r+1, k}$ is true for $1 \leq k \leq s-1$.\\
We will show the equation $\mathrm{E}_{r+1, s}$. For a positive integer $i \in \{ 1, \ldots, r-s+2\}$, we consider the set $\{ (\tau_{1}, \ldots, \tau_{s}) \in T_{s, j} | \tau_{1} = i \}$. Then, we easily count the number $\# \{ (\tau_{1}, \ldots, \tau_{s}) \in T_{s, j} | \tau_{1} = i \} = \binom {r-i}{s-2}$. By applying the induction hypothesis $(\mathrm{A})$ and $(\mathrm{B})$ to $(\tau_{2}, \ldots, \tau_{s})$, we have
{\small\[
\widetilde{C}^{(s)}_{A_l, j} = \sum_{i=1}^{r-s+2} \biggl\{ \binom {i+1}{2}\binom {r-i}{s-2} + \widetilde{C}^{(s-1)}_{A_l, r+1-i} \biggr\}\,\,\,\,\,\,\,\,\,\,\,\,\,\,\,\,\,\,\,\,\,\,\,\,\,\,\,\,\,\,\,\,\,\,\,\,\,\,\,\,\,\,\,\,\,\,\,\,
\]}
{\small\[
= \sum_{i=1}^{r-s+2} \binom {i+1}{2}\binom {r-i}{s-2} + (s-1) \sum_{i=1}^{r-s+2} \binom {r+2-i}{s}
\]}
{\small\[
= \binom {r+2}{s+1} + (s-1)\binom {r+2}{s+1}= s\binom {r+2}{s+1}.\,\,\,\,\,\,\,\,\,\,\,\,\,\,\,\,\,\,\,\,\,\,\,\,\,\,\,\,
\]}
This completes the proof.
\end{proof}
\begin{theorem}
{\it For positive integers $l, j$ with $l \ge j$, the following equation holds:
{\small\[
C_{A_{l+1}, j+1} - C_{A_{l}, j}= (j+1)\binom {l+1}{j+1}.
\]}
}
\end{theorem}
\begin{proof} 
First, we remark the following.
\begin{proposition}
{\it Therem 3.6 implies Therem 1.1 for the type $A_{l}$.} 
\end{proposition}
\begin{proof} 
It is easy to show $N'_{A_{1}}(1)=-1$. Hence, it suffices to show that\\
 $N'_{A_{l+1}}(1) + N'_{A_{l}}(1)=0$ for any positive integer $l$.
{\small\[
N'_{A_{l+1}}(1) + N'_{A_{l}}(1)\,\,\,\,\,\,\,\,\,\,\,\,\,\,\,\,\,\,\,\,\,\,\,\,\,\,\,\,\,\,\,\,\,\,\,\,\,\,\,\,\,\,\,\,\,\,\,\,\,\,\,
\]}
 \vspace{-0.15cm}
{\small\[
\,\,\,\,\,\,\,\,\,\,\,\,\,\,\,\,\,= -C_{A_{l+1}, 1} + \sum_{j=1}^{l} (C_{A_{l+1}, j+1}- C_{A_{l}, j})(-1)^{j+1}
\]}
 \vspace{-0.15cm}
{\small\[
\,\,\,\,\,\,= -C_{A_{l+1}, 1} + \sum_{j=1}^{l} (-1)^{j+1}(j+1)\binom {l+1}{j+1}
\]}
 \vspace{-0.15cm}
{\small\[
= \sum_{j=0}^{l} (-1)^{j+1}(j+1)\binom {l+1}{j+1}\,\,\,\,\,\,\,\,\,\,\,\,\,\,\,\,\,\,\,\,\,\,\,\,\,
\]}
 \vspace{-0.15cm}
{\small\[
\,\,= -(l+1)\sum_{j=0}^{l} (-1)^{j}\binom {l}{j}\,\,\,\,\,\,\,\,\,\,\,\,\,\,\,\,\,\,\,\,\,\,\,\,\,\,\,\,\,\,\,\,\,\,\,\,\,\,
\]}
 \vspace{-0.15cm}
{\small\[
=- \left.\frac{\mathrm{d}}{\mathrm{d}x}(1+x)^{l+1}\right|_{x=-1} =0.\,\,\,\,\,\,\,\,\,\,\,\,\,\,\,\,\,\,\,\,\,\,\,\,\,
\]}
\end{proof}
\noindent 
To prove Theorem 3.6, we prepare a lemma.
\begin{lemma}
{\it   
 (1) For positive integers $l, j$ with $j+1 \leq \lceil \frac{l+1}{2} \rceil$, the following equation holds:
{\small\[
\sum_{k=1}^{j+1} \binom {j}{k-1}\binom {l-j+1}{k} = \binom {l+1}{j+1}.
\]}
\\
(2) For positive integers $l, j$ with $j+1 > \lceil \frac{l+1}{2} \rceil$, the following equation holds.
{\small\[
\sum_{k=1}^{l-j+1} \binom {j}{k-1}\binom {l-j+1}{k} = \binom {l+1}{j+1}.
\]}
 }
\end{lemma}
\begin{proof} 
(1) We rewrite the equation as follows
{\small\[
 \sum_{k=1}^{j+1} \binom {j}{k-1}\binom {l-j+1}{l-j+1-k} = \binom {l+1}{l-j}.
\]}
 This follows from $(1+x)^{l+1}=(1+x)^{j}(1+x)^{l-j+1}$.\\
(2) In the same way, we obtain the result.
\end{proof}

\noindent
We consider two cases.\par
Case 1\,: $j+1 \leq \lceil \frac{l+1}{2} \rceil$.\\
{\small\[
C_{A_{l+1}, j+1} - C_{A_l, j} \,\,\,\,\,\,\,\,\,\,\,\,\,\,\,\,\,\,\,\,\,\,\,\,\,\,\,\,\,\,\,\,\,\,\,\,\,\,\,\,\,\,\,\,\,\,\,\,\,\,\,\,\,\,\,\,\,\,\,\,\,\,\,\,\,\,\,\,\,\,\,\,\,\,\,\,\,\,\,\,\,\,\,\,\,\,\,\,\,\,\,\,\,\,\,\,\,\,\,\,\,
\]}
{\small\[
= \sum_{k=1}^{j+1} C^{(k)}_{A_{l+1}, j+1} - \sum_{k=1}^{j} C^{(k)}_{A_l, j}\,\,\,\,\,\,\,\,\,\,\,\,\,\,\,\,\,\,\,\,\,\,\,\,\,\,\,\,\,\,\,\,\,\,\,\,\,\,\,\,\,\,\,\,\,\,\,\,\,\,\,\,\,\,\,\,\,\,\,\,\,\,\,\,\,\,\,\,\,\,\,\,\,\,\,\,\,\,
\]}
{\small\[
=\sum_{k=1}^{j+1} k\binom {j+2}{k+1}\binom {l-j+1}{k}  - \sum_{k=1}^{j} k \binom {j+1}{k+1}\binom {l-j+1}{k}
\]}
{\small\[
= (j+1)\binom {l-j+1}{j+1} + \sum_{k=1}^{j} k\binom {j+1}{k}\binom {l-j+1}{k}\,\,\,\,\,\,\,\,\,\,\,\,\,\,\,\,\,\, 
\]}
{\small\[
=  \sum_{k=1}^{j+1} k\binom {j+1}{k}\binom {l-j+1}{k}\,\,\,\,\,\,\,\,\,\,\,\,\,\,\,\,\,\,\,\,\,\,\,\,\,\,\,\,\,\,\,\,\,\,\,\,\,\,\,\,\,\,\,\,\,\,\,\,\,\,\,\,\,\,\,\,\,\,\,\,\,\,\,\,\,\,\,\,\,\,\,\,\,\,
\]}
{\small\[
= (j+1)\sum_{k=1}^{j+1} \binom {j}{k-1}\binom {l-j+1}{k}.\,\,\,\,\,\,\,\,\,\,\,\,\,\,\,\,\,\,\,\,\,\,\,\,\,\,\,\,\,\,\,\,\,\,\,\,\,\,\,\,\,\,\,\,\,\,\,\,\,\,\,\,\,\,\,\,
\]}
Thanks to the Lemma 3.8 (1), we have
{\small\[
C_{A_{l+1}, j+1} - C_{A_l, j}= (j+1)\binom {l+1}{j+1}.
\]}
\par
Case 2\,: $j+1 > \lceil \frac{l+1}{2} \rceil$.\\
{\small\[
C_{A_{l+1}, j+1} - C_{A_{l}, j}\,\,\,\,\,\,\,\,\,\,\,\,\,\,\,\,\,\,\,\,\,\,\,\,\,\,\,\,\,\,\,\,\,\,\,\,\,\,\,\,\,\,\,\,\,\,\,\,\,\,\,\,\,\,\,\,\,\,\,\,\,\,\,\,\,\,\,\,\,\,\,\,\,\,\,\,\,\,\,\,\,\,\,\,\,\,\,\,\,\,\,\,\,\,\,\,\,\,\,\,\,\,\,\,\,\,\,\,\,\,\,\,
\]}
{\small\[
= \sum_{k=1}^{l-j+1} C^{(k)}_{A_{l+1}, j+1} - \sum_{k=1}^{l-j+1} C^{(k)}_{A_l, j}\,\,\,\,\,\,\,\,\,\,\,\,\,\,\,\,\,\,\,\,\,\,\,\,\,\,\,\,\,\,\,\,\,\,\,\,\,\,\,\,\,\,\,\,\,\,\,\,\,\,\,\,\,\,\,\,\,\,\,\,\,\,\,\,\,\,\,\,\,\,\,\,\,\,\,\,\,
\]}
{\small\[
=\sum_{k=1}^{l-j+1} k\binom {j+2}{k+1}\binom {l-j+1}{k}  - \sum_{k=1}^{l-j+1} k \binom {j+1}{k+1}\binom {l-j+1}{k}
\]}
{\small\[
= \sum_{k=1}^{l-j+1} k\binom {j+1}{k}\binom {l-j+1}{k} \,\,\,\,\,\,\,\,\,\,\,\,\,\,\,\,\,\,\,\,\,\,\,\,\,\,\,\,\,\,\,\,\,\,\,\,\,\,\,\,\,\,\,\,\,\,\,\,\,\,\,\,\,\,\,\,\,\,\,\,\,\,\,\,\,\,\,\,\,\,\,\,\,\,\,\,\,\,\,\,
\]}
{\small\[
= (j+1)\sum_{k=1}^{l-j+1} \binom {j}{k-1}\binom {l-j+1}{k}.\,\,\,\,\,\,\,\,\,\,\,\,\,\,\,\,\,\,\,\,\,\,\,\,\,\,\,\,\,\,\,\,\,\,\,\,\,\,\,\,\,\,\,\,\,\,\,\,\,\,\,\,\,\,\,\,\,\,\,\,\,\,
\]}
Thanks to the Lemma 3.8 (2), we have
{\small\[
C_{A_{l+1}, j+1} - C_{A_l, j}= (j+1)\binom {l+1}{j+1}.
\]}
This completes the proof of Theorem 3.6.
\end{proof} 
\section{Proof of the type $B_{l}$}
Let $l$ be a positive integer in $\Z_{\ge2}$ and let $I = \{1, 2, \ldots , l \}$. In this section, we will prove Theorem 1.1 for the type $B_{l}$. We recall a fact from \cite{[B-S]}.
\begin{proposition}
{\it For a subset $J \subset I$, we suppose that the full subgraph  $\Gamma_{B_{l}}(J)$ is connected. Then, the degree $\deg(\Delta_{B_{l}, J})$ of the fundamental element is given by
\[
\deg(\Delta_{B_{l}, J}) = \left \{ \begin{array}{lll}
\#(J)^2  & \mbox{if $J \supset \{ 1, 2 \}$}  \\
 \binom {\#(J)+1}{2} & \mbox{if  $J \not\supset \{ 1, 2 \}$} \\
   
          \end{array}
\right.
\]
}
\end{proposition}
\noindent
From the equation (2.1), if the full subgraph  $\Gamma_{B_{l}}(J)$ for a subset $J \subset I$ has a decomposition $\Gamma_{B_{l}}(J) = \bigsqcup_{i=1}^{k} \Gamma_{B_{l}}(J_i)$ with $\mathrm{min}(J_{1}) < \cdots < \mathrm{min}(J_{k})$, then we can compute the degree $\deg(\Delta_{B_{l}, J})$ of the fundamental element.
In the case of $J_1 \not\supset \{ 1, 2 \}$, we compute 
\begin{equation}
\deg(\Delta_{B_{l}, J}) = \sum_{i=1}^{k} \deg(\Delta_{B_{l}, J_i}) = \sum_{i=1}^{k} \binom {\#(J_i)+1}{2}. 
\end{equation}
Moreover, in the case of $J_1 \supset \{ 1, 2 \}$, we compute 
\begin{equation}
\deg(\Delta_{B_{l}, J}) = \sum_{i=1}^{k} \deg(\Delta_{B_{l}, J_i}) = \#(J_1)^2 + \sum_{i=2}^{k} \binom {\#(J_i)+1}{2}. 
\end{equation}
\begin{theorem}
{\it The following equation holds:
{\small\[
N'_{B_{l}}(1) - N'_{A_{l}}(1) = (-1)^{l}(l-1).
\] }
}
\end{theorem}
\begin{proof} 
We compute the difference between $N'_{B_{l}}(1)$ and $N'_{A_{l}}(1)$. From the equations (4.1) and (4.2), we only have to count the case when the set $J_{1}$ contains the index set $\{ 1, 2 \}$. For a positive integer $u \in \{2, \ldots, l-2 \}$ and $X_l \in \{ A_{l}, B_{l}\}$, we put
\begin{equation*}
S_{X_{l}, {u}} := \left\{  J \subset I \left|
\begin{array}{l}
\Gamma_{X_{l}}(J) = \bigsqcup_{i=1}^{k} \Gamma_{X_{l}}(J_i)\,\,\mathrm{with}\,\,\mathrm{min}(J_{1}) < \cdots < \mathrm{min}(J_{k})\\
\mathrm{s.t.}\,\, J_1 = \{ 1, \ldots, u\} \\
\end{array}
\right.\right\}
\end{equation*}
For each $u \in \{2, \ldots, l-2 \}$, the difference on $S_{X_{l}, u}$ is the following
{\small\[
\sum_{J \in S_{B_{l}, u}}(-1)^{\#J} \deg(\Delta_{B_{l}, J}) - \sum_{J \in S_{A_{l}, u}}(-1)^{\#J} \deg(\Delta_{A_{l}, J}) \,\,\,\,\,\,\,\,\,\,
\]}
{\small\[
= \sum_{J \in S_{B_{l}, u}}(-1)^{\#J} \biggl\{\deg(\Delta_{B_{l}, J}) - \deg(\Delta_{A_{l}, J})\biggr\} \,\,\,\,\,\,\,\,\,\,\,\,\,\,\,\,\,\,\,\,\,\,\,\,\,\,\,\,\,\,\,
\]}
{\small\[
= \sum_{J \in S_{B_{l}, u}}(-1)^{\#J}\biggl\{ u^2 - \binom {u+1}{2} \biggr\}\,\,\,\,\,\,\,\,\,\,\,\,\,\,\,\,\,\,\,\,\,\,\,\,\,\,\,\,\,\,\,\,\,\,\,\,\,\,\,\,\,\,\,\,\,\,\,\,\,\,\,\,\,\,\,\,\,\,\,
\]}
{\small\[
 = \binom {u}{2} \sum_{J \in S_{B_{l}, u}}(-1)^{\#J}  =0.\,\,\,\,\,\,\,\,\,\,\,\,\,\,\,\,\,\,\,\,\,\,\,\,\,\,\,\,\,\,\,\,\,\,\,\,\,\,\,\,\,\,\,\,\,\,\,\,\,\,\,\,\,\,\,\,\,\,\,\,\,\,\,\,\,\,\,\,\,\,\,\,\,\,\,\,\,\,
\]}
Hence, we only have to count the cases $J_1 = \{ 1, \ldots, l-1\}, \{ 1, \ldots, l\}$
{\small\[
N'_{B_{l}}(1) - N'_{A_{l}}(1) \,\,\,\,\,\, \,\,\,\,\,\, \,\,\,\,\,\, \,\,\,\,\,\,\,\,\,\,\,\,\,\,\,\,\,\,\,\,\,\,\,\,\,\,\,\,\,\,\,\,\,\,\,\,\,\,\,\,\,\,\,\,\,\,\,\,\,\,\,\,\,\,\,\,\,\,\,\,\,\,\,\,\,\,\,\,\, 
\]}
{\small\[
\,\,\,\,\,\,\,\,= (-1)^{l-1}\biggl\{ (l-1)^2 - \binom {l}{2} \biggr\} + (-1)^{l}\biggl\{ l^2 - \binom {l+1}{2} \biggr\} 
\]}
{\small\[
= (-1)^{l}(l-1).\,\,\,\,\,\,\,\,\,\,\,\,\,\,\,\,\,\,\,\,\,\,\,\,\,\,\,\,\,\,\,\,\,\,\,\,\,\,\,\,\,\,\,\,\,\,\,\,\,\,\,\,\,\,\,\,\,\,\,\,\,\,\,\,\,\,\,\,\,\,\,\,\,\,\,\,\,\,\,\,\,\,\,\,\,\,\,\,\,\,\,\,
\]}
This completes the proof.
\end{proof} 
\section{Proof of the type $D_{l}$}
Let $l$ be a positive integer in $\Z_{\ge4}$ and let $I = \{1, 2, \ldots , l \}$. In this section, we will prove Theorem 1.1 for the type $D_{l}$. We recall a fact from \cite{[B-S]}.
\begin{proposition}
{\it For a subset $J \subset I$, we suppose that the full subgraph  $\Gamma_{D_{l}}(J)$ is connected. Then, the degree $\deg(\Delta_{D_{l}, J})$ of the fundamental element is given by
\[
\deg(\Delta_{D_{l}, J}) = \left \{ \begin{array}{lll}
 \binom {\#(J)+1}{2} & \mbox{if $J \not\supset \{ l-1, l \}$}  \\
 \#(J)(\#(J)-1) & \mbox{if  $J \supset \{ l-3, l-2, l-1, l \}$} \\
  6 & \mbox{if  $J = \{ l-2, l-1, l \}$}
          \end{array}
\right.
\]
}
From the equation (2.1), if the full subgraph  $\Gamma_{D_{l}}(J)$ for a subset $J \subset I$ has a decomposition $\Gamma_{D_{l}}(J) = \Gamma_{D_{l}}(J_1) \sqcup \cdots \sqcup \Gamma_{D_{l}}(J_k)$ with $\mathrm{min}(J_{1}) < \cdots < \mathrm{min}(J_{k})$, then we can compute the degree $\deg(\Delta_{D_{l}, J})$ of the fundamental element.
\end{proposition}
\begin{theorem}
{\it The following equality holds:
{\small\[
N'_{D_{l}}(1)  = (-1)^{l}(l-2).
\] }
}
\end{theorem}
\begin{proof} 
We will show the statement by induction on $l$. First, for $l=4$, we easily compute $N'_{D_{4}}(1) = \sum_{J \subset I}(-1)^{\#J}\deg(\Delta_{D_{4}, J}) = 12$. Next, by applying the induction hypothesis, we will compute the difference $N'_{D_{l}}(1) - N'_{D_{l-1}}(1)$. For a positive integer $u \in \{1, \ldots, l-2 \}$, we put
\begin{equation*}
S_{D_{l}, {u}} := \left\{  J \subset I \left|
\begin{array}{l}
\Gamma_{D_{l}}(J) = \bigsqcup_{i=1}^{k} \Gamma_{D_{l}}(J_i)\,\,\mathrm{with}\,\,\mathrm{min}(J_{1}) < \cdots < \mathrm{min}(J_{k})\\
\mathrm{s.t.}\,\, J_1 = \{ 1, \ldots, u\} \\
\end{array}
\right.\right\}
\end{equation*}

\[
I_{u} := I\setminus \{1, \ldots, u, u+1 \}.\,\,\,\,\,\,\,\,\,\,\,\,\,\,\,\,\,\,\,\,\,\,\,\,\,\,\,\,\,\,\,\,\,\,\,\,\,\,\,\,\,\,\,\,\,\,\,\,\,\,\,\,\,\,\,\,\,\,\,\,\,\,\,\,\,\,\,\,\,\,\,\,\,\,\,\,\,\,\,\,\,\,\,\,\,\,\,\,\,\,\,\,\,\,\,\,\,\,\,\,\,\,\,\,\,\,\,\,\,\,\,\,\,\,\,\,\,\,\,\,\,\,\,\,\,\,\,
\]
For each $u \in \{1, \ldots, l-5 \}$, the difference on $S_{D_{l}, u}$ is the following
{\small\[
\sum_{J \in S_{D_{l}, u}}(-1)^{\#J} \deg(\Delta_{D_{l}, J}) \,\,\,\,\,\,\,\,\,\,\,\,\,\,\,\,\,\,\,\,\,\,\,\,\,\,\,\,\,\,\,\,\,\,\,\,\,\,\,\,\,\,\,\,\,\,\,\,\,\,\,\,\,\,\,\,\,\,\,\,\,\,\,\,\,\,\,\,\,\,\,\,\,\,\,\,\,\,\,\,\,\,\,\,\,\,\,\,\,\,\,\,\,\,\,\,\,\,\,\,\,\,\,\,\,\,\,\,\,
\]}
{\small\[
=(-1)^{u}\sum_{K \subset I_{u}}(-1)^{\#K} \biggl\{ \deg(\Delta_{D_{l}, J_{1}}) + \deg(\Delta_{D_{l}, K}) \biggr\} \,\,\,\,\,\,\,\,\,\,\,\,\,\,\,\,\,\,\,\,\,\,\,\,\,\,\,\,\,\,\,\,\,\,\,\,\,\,\,\,\,\,\,\,
\]}
{\small\[
=(-1)^{u} \deg(\Delta_{D_{l}, J_{1}})\sum_{K \subset I_{u}}(-1)^{\#K} + (-1)^{u}\sum_{K \subset I_{u}} (-1)^{\#K}\deg(\Delta_{D_{l}, K}) 
\]}
{\small\[
= (-1)^{u}\sum_{K \subset I_{u}} (-1)^{\#K}\deg(\Delta_{D_{l}, K}) \,\,\,\,\,\,\,\,\,\,\,\,\,\,\,\,\,\,\,\,\,\,\,\,\,\,\,\,\,\,\,\,\,\,\,\,\,\,\,\,\,\,\,\,\,\,\,\,\,\,\,\,\,\,\,\,\,\,\,\,\,\,\,\,\,\,\,\,\,\,\,\,\,\,\,\,\,\,\,\,\,\,\,\,\,\,\,\,\,\,
\]}
{\small\[
= (-1)^{u}N'_{D_{l-u-1}}(1). \,\,\,\,\,\,\,\,\,\,\,\,\,\,\,\,\,\,\,\,\,\,\,\,\,\,\,\,\,\,\,\,\,\,\,\,\,\,\,\,\,\,\,\,\,\,\,\,\,\,\,\,\,\,\,\,\,\,\,\,\,\,\,\,\,\,\,\,\,\,\,\,\,\,\,\,\,\,\,\,\,\,\,\,\,\,\,\,\,\,\,\,\,\,\,\,\,\,\,\,\,\,\,\,\,\,\,\,\,\,\,\,\,\,\,\,\,\,\,\,\,\,\,\,
\]}
From the induction hypothesis, this is equal to $(-1)^{l-1}(l-u-3)$.
For $u = l-4$, the difference on $S_{D_{l}, l-4}$ is computed in a similar manner
{\small\[
\sum_{J \in S_{D_{l}, l-4}}(-1)^{\#J} \deg(\Delta_{D_{l}, J}) \,\,\,\,\,\,\,\,\,\,\,\,\,\,\,\,\,\,\,\,\,\,\,\,\,\,\,\,\,\,\,\,\,\,\,\,\,\,\,\,\,\,\,\,\,\,\,\,\,\,\,\,\,\,\,\,\,\,\,\,\,\,\,\,\,\,\,\,\,\,\,\,\,\,\,\,\,\,\,\,\,\,\,\,\,\,\,\,\,\,\,\,\,
\]}
{\small\[
= (-1)^{l-4}N'_{A_{3}}(1). \,\,\,\,\,\,\,\,\,\,\,\,\,\,\,\,\,\,\,\,\,\,\,\,\,\,\,\,\,\,\,\,\,\,\,\,\,\,\,\,\,\,\,\,\,\,\,\,\,\,\,\,\,\,\,\,\,\,\,\,\,\,\,\,\,\,\,\,\,\,\,\,\,\,\,\,\,\,\,\,\,\,\,\,\,\,\,\,\,\,\,\,\,\,\,\,\,\,\,\,\,\,\,\,\,\,\,\,\,\,
\]}
For $u = l-3$, we easily compute
{\small\[
\sum_{J \in S_{D_{l}, l-3}}(-1)^{\#J} \deg(\Delta_{D_{l}, J})=0. \,\,\,\,\,\,\,\,\,\,\,\,\,\,\,\,\,\,\,\,\,\,\,\,\,\,\,\,\,\,\,\,\,\,\,\,\,\,\,\,\,\,\,\,\,\,\,\,\,\,\,\,\,\,\,\,\,\,\,\,\,\,\,\,\,\,\,\,\,\,\,\,\,\,
\]}
Therefore, we can compute the difference
{\small\[
N'_{D_{l}}(1) - N'_{D_{l-1}}(1)\,\,\,\,\,\,\,\,\,\,\,\,\,\,\,\,\,\,\,\,\,\,\,\,\,\,\,\,\,\,\,\,\,\,\,\,\,\,\,\,\,\,\,\,\,\,\,\,\,\,\,\,\,\,\,\,\,\,\,\,\,\,\,\,\,\,\,\,\,\,\,\,\,\,\,\,\,\,\,\,\,\,\,\,\,\,\,\,\,\,\,\,\,\,\,\,\,\,
\]}
{\small\[
=\sum_{u=1}^{l-3} \sum_{J \in S_{D_{l}, u}}(-1)^{\#J} \deg(\Delta_{D_{l}, J}) + (-1)^{l-2}\binom {l-1}{2}\,\,\,\,\,\,\,\,\,\,\,\,\,\,\,\,\,
\]}
{\small\[
=(-1)^{l-1} \biggl\{ (l-4)+(l-3)+ \cdots +1 \biggr\} + (-1)^{l-2}\binom {l-1}{2}
\]}
{\small\[
= (-1)^{l}(2l-5).\,\,\,\,\,\,\,\,\,\,\,\,\,\,\,\,\,\,\,\,\,\,\,\,\,\,\,\,\,\,\,\,\,\,\,\,\,\,\,\,\,\,\,\,\,\,\,\,\,\,\,\,\,\,\,\,\,\,\,\,\,\,\,\,\,\,\,\,\,\,\,\,\,\,\,\,\,\,\,\,\,\,\,\,\,\,\,\,\,\,\,\,\,\,\,\,\,\,\,\,\,\,\,\,\,\,\,\,
\]}
From the induction hypothesis, we have
{\small\[
N'_{D_{l}}(1)=(-1)^{l}(l-2).
\]}
\end{proof}
\noindent
\emph{Acknowledgement.}\! 
The author thanks Kyoji Saito for very interesting discussions and encouragement. The author is grateful to Toshitake Kohno for his encouragement. This research is supported by JSPS Fellowships for Young Scientists $(24\cdot10023)$. This researsh is also supported by World Premier International Research Center Initiative (WPI Initiative), MEXT, Japan.

\begin{flushright}
\begin{small}
Kavli IPMU, \\
University of Tokyo, \\
Kashiwa, Chiba 277-8583 Japan \\

e-mail address :  tishibe@ms.u-tokyo.ac.jp
\end{small}
\end{flushright}
\end{document}